\documentclass[12pt]{article}
\usepackage{pslatex}
\usepackage{amsthm, amsfonts, amssymb, stmaryrd, amsmath, diagrams}
\usepackage{rotating}
\usepackage{tabularx}
\usepackage{algorithmic}
\usepackage{footmisc}
\usepackage{dsfont}

\newtheorem{theorem}{Theorem}

\newtheorem{lemma}[theorem]{Lemma}
\newtheorem{prop}[theorem]{Proposition}
\newtheorem{corollary}[theorem]{Corollary}
\newtheorem{defn}[theorem]{Definition}
\theoremstyle{definition}
\newtheorem{remark}[theorem]{Remark}

\def\R{\mathbb{R}}

\def\NN{\mathcal N}

\def\N{\mathbb{N}}

\title{Persistent Homology of Filtered Covers}
\author{\large Maia Fraser\thanks{Department of Computer Science, University of Chicago 
{\tt maia@cs.uchicago.edu}. 
} }
\date{February 27, 2012}
\begin{document}

\maketitle

\begin{abstract}
We prove an extension to the simplicial
Nerve Lemma which establishes isomorphism of persistent homology groups,
in the case where the covering spaces are filtered. While persistent homology is
now widely used in topological data analysis, the usual Nerve Lemma does not
provide isomorphism of persistent homology groups.
Our argument involves some homological algebra: the key point being
that although the maps produced in the standard proof of the Nerve
Lemma do not commute as maps of chain complexes, the maps they induce
on homology do. 
\end{abstract}

Persistent homology has become a central tool in topological data analysis (TDA).
The purpose of the present paper is to update the Nerve Lemma accordingly\footnote{This result was announced in the tech report \cite{fraser}, which also surveys relevant recent work in TDA;  however, the sketch of proof suggested there did not correctly anticipate the issue discussed in Remark~\ref{rem:noCommute}.}.
Specifically we prove an extension to the
(finite) simplicial version of the Nerve Lemma, which is sufficient for the usual TDA applications.
Our proof is self-contained and elementary. After the writing of this paper, it was brought to our attention that
Chazal and Oudot \cite{chazalOudot} have proved an analogous result 
in the topological category, however their proof is more involved and relies on earlier work. 
In this light we feel there is merit in publishing our own argument for two reasons: 1. it provides
a proof in the simplicial category of this basic result (updating the semi-classical Nerve Lemma) 
and 2.  its simplicity allows one to 
see explicitly why the relevant chain maps must 
commute on homology level (see Remark~\ref{rem:noCommute} for a discussion of this issue). 

Our main result is the following.
\begin{prop}\label{prop:main}
Let $\Delta$ be a simplicial complex and 
$\{\Delta^\ell \}_{\ell \in I}$ a filtration of $\Delta$, as a topological space, such that
each $\Delta^\ell$ is also a simplicial complex. 
For each $\ell \in I$, suppose $(\Delta^{\ell}_i)_{i \in I}$ is a locally finite family of subcomplexes of  $\Delta^\ell$.
such that $\Delta^\ell = \bigcup\limits_{i \in I} \Delta^\ell_i$ and every nonempty finite intersection
$\Delta^\ell_{i_1} \cap \cdots, \cap \Delta^\ell_{i_t}$ is contractible. 
Suppose now for $\ell,  \ell+p \in I, p>0$ we have
\begin{equation} \label{eq:goodFiltration}
\Delta^\ell_i \text{ is a subcomplex of } \Delta^{\ell+p}_i.
\end{equation}
Then for each $k \in \N$,
$$H_k^p(\Delta^\ell) \cong H_k^p(\NN(\Delta^\ell_i)),$$
where $\NN(\Delta^\ell_i)$ is the nerve of the collection $(\Delta^{\ell}_i)_{i \in I}$. In other words
the $p$-persistent homology groups coincide at level $\ell$ 
in the filtrations $\Delta^\ell \subset \Delta^{\ell+p}$ and $\NN(\Delta^\ell_i) \subset \NN(\Delta^{\ell+p}_i) $.
\end{prop}

We define persistent homology below. For more detail the reader is referred to
Weinberger's short expository article \cite{weinberger}, or Zomorodian's thesis \cite{zomor}.  

\paragraph{\sc Application to topological data analysis}
Much recent TDA work \cite{vinAlgTop, Edel02, EdelSurvey, noisyChazal, fundamChazal, 
glisseChazal, Edel07, deSilvaGunnar, oudotWitness, zomor, polyRings} uses $\alpha$-complexes 
to recover topological invariants of a submanifold
$\mathcal M \subset \R^N$ from a point cloud $Z$ associated to $\mathcal M$  while \cite{NSW1, NSW2} 
(in the case of smooth $\mathcal M$) uses \v{C}ech complexes and addresses more general sampling.  

Given a (finite) point cloud $Z \subset \R^d$, denote $K(Z, \alpha)$ the $\alpha$-complex, first defined by 
Edelsbrunner in 1995, \cite{EdelAlpha}; denote by \v{C}ech$(Z, \alpha)$ the \v{C}ech complex. In each case, 
vertices are points of $Z$. The {\bf \v{C}ech complex} is defined by:

$$\sigma = [z_0z_1\ldots z_p] \text{ is a $p$-simplex iff }\bigcap\limits_{j = 0}^p B(z_j, \alpha) \neq \emptyset$$. 

\noindent
The {\bf $\alpha$-complex $A(Z, \alpha)$} is defined\footnote{Warning: usually the notation $A(Z, \alpha)$ is used to 
denote the analogous complex obtained using balls of radius $\alpha/2$ intersected with Voronoi cells.\label{foot:halfalpha}} by:

$$\sigma = [z_0z_1\ldots z_p]\text{ is a $p$-simplex iff }
\bigcap\limits_{j = 0}^p [ B(z_j, \alpha) \cap V_j ] \neq \emptyset,$$
where $V_j = \{x \in \R^N: \forall z \in Z,\ d(x, z_j) \leq d(x, z) \}$ is the Voronoi cell of $z_j \in Z$. 

Let $U$ be the union of balls of radius $\alpha$ around points of $Z$. Both of the collections of  sets --
$\{B(z_j, \alpha): z_j \in Z \}$ and $\{ B(z_j, \alpha) \cap V_j : z_j \in Z \}$ -- are finite covers of $U$ and the 
complexes just defined are their {\bf nerves} (see Hatcher \cite{hatcher} or Bjoerner \cite{bjoerner}). 
Moreover the sets in these covers are convex. The Nerve Lemma \cite{bjoerner} therefore implies the 
nerves are homotopy equivalent to $U$, and hence to each other. 

\bigskip
 Proposition~\ref{prop:main}, as an extension of the simplicial Nerve Lemma, shows that the persistent homology 
groups coincide. Indeed for 
\v{C}ech complexes we may triangulate the collection of larger 
balls\footnote{The radius $\alpha$ plays the role of $\ell$ in the Proposition.} so that the smaller balls and intersections are 
subcomplexes, and similarly for $\alpha$-complexes; their union in either case is $U$.

\paragraph{\sc Persistent Homology}
Using the standard notation $C_k(X), Z_k(X), B_k(X)$ for $k$-chains, $k$-cycles and 
$k$-boundaries respectively of a simplicial complex $X$, we recall:

\begin{defn}[Persistent Homology]
Given integers $p, k >0$, and a filtered topological space $X = \bigcup\limits_{\ell=0}^\infty X^\ell$ with
$\ell < \ell' \Rightarrow X^\ell \subset X^{\ell'}$. The $p$-persistent $k$-th homology of $X^\ell$ is
the image of $H_k(X^\ell)$  in $H_k(X^{\ell+p})$ induced by inclusion. Equivalently, it may be defined as
\begin{align*}
H_k^p(X^\ell) := \frac {Z_k(X^\ell)} {B_k(X^{\ell+p}) \cap Z_k(X^\ell)}.
\end{align*}
\end{defn}

\paragraph{\sc Details of the Simplicial Nerve Lemma: Posets and Order Complexes}
This section is a summary of the relevant exposition in Bjoerner \cite{bjoerner}.
We use {\it poset} as a shorthand for {\it partially ordered set}.

The {\bf face poset} $P(\Delta)$ of a simplicial complex $\Delta$ is the set of faces (simplices) 
of $\Delta$ ordered by inclusion.
The {\bf order complex} $\Delta(P)$ of a poset $P$ with partial order $\leq$ is the simplicial complex with 
vertex set $P$ such that  $[x_0 \ldots x_n]$ a $k$-simplex if and only if $x_0 < \ldots < x_k$.
Given a simplicial complex $\Delta$, the simplicial complex $\Delta(P(\Delta))$ is called the {\bf barycentric subdivision}
of $\Delta$; it is homeomorphic to $\Delta$ (using geometric realizations). For readability we write $\Delta P(\Delta)$ .

From now on, we {\bf assume the hypotheses of Proposition~\ref{prop:main}.}
The simplicial version of the Nerve Lemma is proved in \cite{bjoerner} by showing that
a certain continuous map 
$$\Theta^\ell: \Delta P(\Delta^\ell) \to \Delta P(\NN(\Delta^\ell_i))$$ 
is a homotopy equivalence, so $\Delta^\ell$ and $\NN(\Delta^\ell_i)$ are homotopy equivalent. In particular,
$\Theta^\ell$ induces an isomorphism between homology groups.
The map $\Theta^\ell$ is defined starting with the poset map $f^\ell: P(\Delta^\ell)) \to P(\NN(\Delta^\ell_i))$ given by
$$\pi \mapsto \{i \in I \colon \pi \in \Delta_i \}.$$
This is an order-reversing poset map and so induces a simplicial map, 
$$\Theta^\ell \colon \Delta P(\Delta^\ell) \to \Delta P(\NN(\Delta^\ell_i)),$$ 
whose effect on vertices is given by $f^\ell$.
In fact $\Theta^\ell$ can also be defined in this way on all of $\Delta P(\Delta^{\ell+p})$, and we will assume this.

\begin{remark} \label{rem:nested}
We remark that $\Delta P(\Delta^\ell)$ is a subcomplex of $\Delta P(\Delta^{\ell+p})$ because 
$\Delta^\ell$ is a subcomplex of $\Delta^{\ell+p}$ (hence any face of 
$\Delta^\ell$ is a face of $\Delta^{\ell+p}$ and moreover nested faces $x_0 < \ldots < x_k$ of $\Delta^\ell$
are nested faces of $\Delta^{\ell+p}$).
Also, $\Delta P (\NN(\Delta_i^\ell))$ is a subcomplex of $\Delta P(\NN(\Delta_i^{\ell+p}))$ by the 
same reasoning, since $\NN(\Delta_i^\ell)$ is a subcomplex of $\NN(\Delta_i^{\ell+p})$ by 
Equation \eqref{eq:goodFiltration}.
Indeed, by \eqref{eq:goodFiltration},
a nonempty intersection $\Delta^\ell_{i_1} \cap \ldots \cap \Delta^\ell_{i_k}$ implies a nonempty
intersection $\Delta^{\ell+p}_{i_1} \cap \ldots \cap \Delta^{\ell+p}_{i_k}$.
We will not write these subcomplex inclusions explicitly; as commented earlier, we assume $\Theta^\ell$ is defined 
on all of $\Delta P(\Delta^{\ell+p})$.
\end{remark}

\begin{remark} \label{rem:noCommute}
Given $\sigma \in \Delta P(\Delta^\ell)$, it is {\it not true} in general that 
$\Theta^{\ell+p} (\sigma) = \Theta^{\ell} (\sigma).$ 
In other words the following diagram does not commute:
\begin{diagram}
\Delta P(\Delta^{\ell+p}) & \rTo^{\Theta^{\ell+p}} & \Delta P(\NN(\Delta_i^{\ell+p})) \\
\uInto^{\subseteq} && \uInto^{\subseteq} \\
\Delta P(\Delta^\ell) &  \rTo^{\Theta^{\ell}} & \Delta P(\NN(\Delta_i^{\ell})).
\end{diagram}
Indeed, this may be seen already at vertex level: the poset map $f^\ell$ takes
a simplex $\pi$ of $\Delta^\ell$ to the set of all indices $i$ such that $\pi$ is a subsimplex of $\Delta^\ell_i$,
while $f^{\ell+p}$ takes $\pi$ to the set of all $i$ such that $\pi$ is a subsimplex of $\Delta^{\ell+p}_i$.
The second set contains the first by Equation \eqref{eq:goodFiltration}, but may be strictly larger.
In that case $\Theta^\ell$ and $\Theta^{\ell+p}$ map
the vertex $\pi$ of $\Delta P(\Delta^\ell) \subset \Delta P(\Delta^{\ell+p})$ 
to distinct vertices of $\Delta P (\NN(\Delta^{\ell+p}_i))$.
\end{remark}

We will, however, show that the induced chain maps $\Theta^\ell_\ast$ and $\Theta^{\ell+p}_\ast$
differ on $k$-cycles of $\Delta P(\Delta^{\ell+p})$
by boundaries of $\Delta P (\NN(\Delta^{\ell+p}_i))$. In other words, the homology-level diagram induced by the above diagram {\it does} commute, giving
an isomorphism of  the respective $p$-persistent homology groups, 
$H_k^p(\Delta P(\Delta^{\ell})) \cong H_k^p(\Delta P (\NN(\Delta^{\ell}_i)))$, and 
therefore $H_k^p(\Delta^{\ell}) \cong H_k^p(\NN(\Delta^{\ell}_i)).$

\paragraph{\sc Technical Lemma}

Given two $k$-simplices $\sigma$ and $\tau$ with a fixed ordering of the vertices of each,
we define a preferred simplicial decomposition of the mapping cylinder of
the simplicial map that sends one simplex to the other preserving vertex order.
Each of the original simplices belongs to this abstract simplicial complex.

\begin{remark} This is a simpler version of the 
usual simplicial mapping cylinder, as we do not take a barycentric subdivision 
of one of the simplices. 
\end{remark}

\noindent
We write $[v_0v_1 \ldots v_k]^o$ to denote the $k$-simplex $[v_0v_1 \ldots v_k]$
with this explicit vertex ordering and refer to it as an {\bf ordered simplex}.

\begin{defn}[Simplicial Mapping Cylinder]
Given two ordered $k$-simplices $\sigma = [v_0 v_1 \ldots v_k]^o$ and $\tau = [w_0 w_1 \ldots w_k]^o$, 
define
\begin{align*}
\text{\rm Cyl}(\sigma, \tau) &:= \sum\limits_{t=0}^k (-1)^{t+1} [v_0 \ldots v_t w_t \ldots w_k],
\end{align*} a formal linear combination of abstract $(k+1)$-simplices on the vertex set 
$\{v_0, \ldots, v_k\} \sqcup \{w_0, \ldots, w_k \}$.
Let $\mu_1, \mu_2$ be $k$-chains of a simplicial complex $X$ with vertex set $V$. If we have
$$ \mu_1 = \sum\limits_{i=0}^m a_i \sigma_i  \text{ and } \mu_2 = \sum\limits_{i=0}^m a_i \tau_i$$
then we say $\mu_1$ and $\mu_2$ are \emph{compatible}
and define (for a fixed ordering of the vertices of each $\sigma_i$ and $\tau_i$)
$$\text{\rm Cyl}(\sum\limits_{i=0}^m a_i \sigma_i, \sum\limits_{i=0}^m a_i \tau_i) := 
\sum\limits_{i=0}^m a_i \, \text{\rm Cyl}(\sigma_i, \tau_i)$$
as a formal linear combination of abstract $(k+1)$-simplices on the vertex set $V \sqcup V$.
\end{defn}

In fact, $\text{Cyl}(\sigma, \tau)$  in the definition, provides a simplicial decomposition of the topological 
mapping cylinder of the map given by $v_i \mapsto w_i$. We will only need the following
(which we prove in the Appendix):
\begin{lemma}\label{lem:cyl}
Given two compatible $k$-chains $\mu_1$ and $\mu_2$,
$$\partial \, \text{\rm Cyl}(\mu_1, \mu_2) = \mu_1 - \mu_2 - \text{\rm Cyl} 
(\partial \mu_1, \partial \mu_2).$$
\end{lemma}

Therefore,
\begin{corollary} \label{cor:cycles}
If $\mu_1$ and $\mu_2$ are compatible $k$-cycles,
$$\partial \, \text{\rm Cyl}(\mu_1, \mu_2) = \mu_1 - \mu_2.$$
\end{corollary}

The reason for defining $\text{Cyl}()$ in this manner is its well-behaved interaction
with $\Theta^\ell$ and $\Theta^{\ell+p}$ which we now describe.
Let $V$ be the vertex set of $\Delta P(\Delta^{\ell+p})$. Suppose we use apostrophes 
to indicate the elements of  $V \sqcup V$ which come from the second factor; so
$V \sqcup V = V \cup \{v' \colon v \in V \}$.  
Let $\mu \in Z_k(\Delta P(\Delta^{\ell+p}))$ be a $k$-cycle. Denote by $\mu'$ the corresponding
$k$-cycle using the vertices $v'$. These are compatible $k$-chains and so $\text{Cyl}(\mu, \mu')$ is well-defined
(for any fixed ordering of the vertices of simplices of $\mu$). 
It is a linear combination of abstract $(k+1)$-simplices on the vertex set $V \sqcup V$ and so we may apply to it 
the chain map $\varphi$ induced by $$v \mapsto f^\ell (v), \;  v' \mapsto f^{\ell+p} (v).$$ 
Both of these images are vertices of $\Delta P(\NN(\Delta^{\ell+p}_i))$. 
By Corollary~\ref{cor:cycles}, we have
\begin{align*}
\partial \varphi \, \text{\rm Cyl}(\mu, \mu')  &=  \varphi \partial  \, \text{\rm Cyl}(\mu, \mu')  \\
&= \varphi ( \mu - \mu' ) \\
&= \Theta^\ell_\ast (\mu) - \Theta^{\ell+p}_\ast (\mu),
\end{align*}
where $\Theta^\ell_\ast$ and $\Theta^{\ell+p}_\ast$ are the chain maps induced 
by $\Theta^\ell$ and $\Theta^{\ell+p}$ respectively. The latter, we assume, are both defined on all of 
$\Delta P(\Delta^{\ell+p})$, mapping into $\Delta P(\NN(\Delta^{\ell+p}_i))$ (see Remark~\ref{rem:nested}). 
Here $\varphi \, \text{\rm Cyl}(\mu, \mu')$ is a formal linear combination of {\it abstract} $(k+1)$-simplices on the 
vertex set of $\Delta P(\NN(\Delta^{\ell+p}_i))$; none of these $(k+1)$-simplices need a priori 
be actual simplices of $\Delta P(\NN(\Delta^{\ell+p}_i))$. 
The following techical lemma shows, however, that they are,
assuming a natural ordering of the vertices in each simplex of $\mu$.

\begin{lemma}\label{lem:tech}
Let $\sigma$ be a $k$-simplex of $\Delta P(\Delta^\ell))$, with the canonical vertex order 
inherited from the underlying poset. Then 
$\varphi \, \text{\rm Cyl}(\sigma, \sigma) \in C_{k+1}(\Delta P(\NN(\Delta^{\ell+p}_i)))$.
Hence, for any $\mu \in Z_k(\Delta P(\Delta^{\ell+p})),$
$\Theta^\ell_\ast (\mu) - \Theta^{\ell+p}_\ast (\mu) \in B_{k}( \Delta P(\NN(\Delta^{\ell+p}_i)))$.
\end{lemma}

\begin{proof}
Note that if $x$ is a vertex of $\Delta P(\Delta^\ell)$ then $x$ is a simplex of $\Delta^\ell$ and by
Remark~\ref{rem:noCommute}, $f^\ell (x) \leq f^{\ell+p}(x)$. Suppose the $k$-simplex $\sigma$ 
of $\Delta P(\Delta^\ell))$ is defined
by nested simplices\footnote{The indexing is done this way to make order-reversed images via $f^\ell$ 
and $f^{\ell+p}$ easier to read.}  $x_k < \ldots < x_0$ of $\Delta^\ell$
and take any $t, \, 1 \leq t \leq k$. We have
$$f^\ell(x_0) \leq \ldots \leq f^\ell(x_t)  \leq f^{\ell+p}(x_t) \leq \ldots f^{\ell+p}(x_k).$$
Therefore, for all $t, \, 1 \leq t \leq k$,
$$\{ f^\ell(x_0), \ldots, f^\ell(x_t), f^{\ell+p}(x_t),\ldots,  f^{\ell+p}(x_k) \}$$ is a simplex of  
$\Delta P( \NN(\Delta^{\ell+p}_i))$ (possibly of dimension less than $k$). 
And so, in the sum for $\text{\rm Cyl}(\sigma, \sigma)$,
the abstract $k$-simplices $[x_0 \ldots x_t x'_t \ldots x'_k]$  which
are not killed off by the chain map $\varphi$ will be mapped to actual $k$-simplices 
$$[f^\ell(x_0) \ldots f^\ell(x_t) f^{\ell+p}(x_t) \ldots f^{\ell+p}(x_k)]$$ of $\Delta P( \NN(\Delta^{\ell+p}_i))$
(apostrophes denoting vertices in the second factor of $V\sqcup V$, as before).
The final statement of the Lemma follows immediately; it suffices to assume the above-mentioned canonical
vertex order in each simplex of $\mu$.
\end{proof}

\paragraph{\sc Proof of the Proposition}

\begin{proof}[Proof of Proposition~\ref{prop:main}.]

We now consider the homology level diagram induced by the diagram of Remark~\ref{rem:noCommute}.
By the proof of the Nerve Lemma, the chain maps
$$\Theta^\ell_\ast: C_k(\Delta P(\Delta^{\ell})) \to C_k( \Delta P(\NN(\Delta^{\ell}_i)))$$ 
and 
$$\Theta^{\ell+p}_\ast: C_k(\Delta P(\Delta^{\ell+p})) \to C_k( \Delta P(\NN(\Delta^{\ell+p}_i)))$$ 
descend to isomorphisms on homology (we retain the same names for the new maps). 
So we have,
\begin{diagram}
H_k(\Delta P(\Delta^{\ell+p})) & \rTo^{\Theta_\ast^{\ell+p}}_\cong & H_k(\Delta P(\NN(\Delta_i^{\ell+p}))) \\
\uTo && \uTo \\
H_k(\Delta P(\Delta^\ell)) &  \rTo^{\Theta_\ast^{\ell}}_\cong & \Delta H_k(P(\NN(\Delta_i^{\ell}))),
\end{diagram}
where the vertical maps are those induced by inclusion. By Lemma~\ref{lem:tech}, {\it this diagram commutes}.
Indeed, given a homology class $[\mu]$ in the bottom left corner, with $\mu$ a cycle representing it, the Lemma implies that $\Theta^\ell_\ast(\mu)$ and $\Theta^{\ell+p}_\ast(\mu)$ differ by a boundary in the top right corner.

Therefore, the image of $H_k(\Delta P(\Delta^\ell))$ in $H_k(\Delta P(\Delta^{\ell+p}))$ is isomorphic to
the image of $H_k(P(\NN(\Delta_i^{\ell})))$ in $H_k(\Delta P(\NN(\Delta_i^{\ell+p})))$; i.e.,
$$H_k^p(\Delta P(\Delta^{\ell})) \cong H_k^p( \Delta P(\NN(\Delta^{\ell}_i))).$$
\end{proof}

\paragraph{\sc Acknowledgements} The author wishes to thank Shmuel Weinberger and Vin de Silva for useful comments on a preliminary version of this paper.

\paragraph{\sc Appendix}
\begin{proof}[Proof of Lemma~\ref{lem:cyl}]
Recall that 
$$\partial \, [v_0 v_1 \ldots v_k] = \sum\limits_{j=0}^k (-1)^j [v_0 \ldots \hat v_j \ldots v_k].$$
We prove the Lemma for $k$-simplices; it follows for compatible $k$-chains.
\begin{align*}
& \hspace{-.5cm} \partial \, \text{\rm Cyl} ([v_0 \ldots v_k], [w_0 \ldots w_k]) \\
&= \partial \sum\limits_{t=0}^k (-1)^{t+1} [v_0 \ldots v_t w_t \ldots w_k] \\
&= -\partial [v_0w_0\ldots w_k] + (-1)^{k+1}\partial [v_0\ldots v_kw_k]
+ \sum\limits_{t=1}^{k-1} (-1)^{t+1} \partial [v_0 \ldots v_t w_t \ldots w_k]\\
&= -[w_0 \ldots w_k] + (-1)^{2(k+1)} [v_0 \ldots v_k] \\
&\hspace{1cm} + \sum\limits_{j=0}^k (-1)^j [v_0w_0\ldots \hat w_j \ldots w_k] 
+ (-1)^{k+1} \sum\limits_{j=0}^k (-1)^j[v_0\ldots \hat v_j \ldots v_kw_k] \\
&\hspace{1cm} + 
\sum_{t=1}^{k-1} (-1)^{t+1} \sum\limits_{j=0}^{t} (-1)^j [v_0 \ldots \hat v_j \ldots v_tw_t\ldots w_k] \\
&\hspace{1cm} + 
\sum_{t=1}^{k-1} (-1)^{t+1} \sum\limits_{j=t}^k (-1)^{j+1} [v_0 \ldots v_tw_t \ldots \hat w_j \ldots w_k] \\
&=  [v_0 \ldots v_k] - [w_0 \ldots w_k] - \text{\rm Cyl} (\partial [v_0 \ldots v_k], \partial [w_0 \ldots w_k]) 
\end{align*}

\noindent
because
\begin{align*}
& \text{\rm Cyl} (\partial [v_0 \ldots v_k], \partial [w_0 \ldots w_k]) \\
&= \sum\limits_{j=0}^k (-1)^j \; \text{\rm Cyl} ([v_0 \ldots \hat v_j \ldots v_k], [w_0\ldots \hat w_j \ldots w_k])  \\
&= \sum\limits_{j=0}^k (-1)^j \left \{ 
\sum\limits_{t=0}^{j-1} (-1)^{t+1} [v_0 \ldots v_t w_t \ldots \hat w_j \ldots w_k] + 
\sum\limits_{t=j+1}^k (-1)^{t} [v_0 \ldots \hat v_j \ldots v_t w_t \ldots w_k] \right \} \\
&= \sum\limits_{t=0}^{k-1} (-1)^j 
\sum\limits_{j=t+1}^{k} (-1)^{t+1} [v_0 \ldots v_t w_t \ldots \hat w_j \ldots w_k] \\
&\hspace{1cm} + 
\sum\limits_{t=1}^{k} (-1)^j 
\sum\limits_{j=0}^{t-1} (-1)^{t} [v_0 \ldots \hat v_j \ldots v_t w_t \ldots w_k] \\
\end{align*}
\noindent
\begin{align*}
&= - \sum\limits_{t=0}^{k-1} 
\sum\limits_{j=t+1}^{k} (-1)^{t+j} [v_0 \ldots v_t w_t \ldots \hat w_j \ldots w_k] \\
&\hspace{1cm} -
\sum\limits_{t=1}^{k} 
\sum\limits_{j=0}^{t-1} (-1)^{t+j+1}  [v_0 \ldots \hat v_j \ldots v_t w_t \ldots w_k] \\
&= -  \left\{ \sum\limits_{t=0}^{k-1} 
\sum\limits_{j=t}^{k} (-1)^{t+j} [v_0 \ldots v_t w_t \ldots \hat w_j \ldots w_k]  - 
\sum\limits_{t=0}^{k-1} (-1)^{2t}[v_0\ldots v_tw_{t+1}\ldots w_k] \right \}\\
&\hspace{1cm} - \left\{ \sum\limits_{t=1}^{k} 
\sum\limits_{j=0}^{t} (-1)^{t+j+1}  [v_0 \ldots \hat v_j \ldots v_t w_t \ldots w_k] -
\sum\limits_{t=1}^{k} (-1)^{2t+1}[v_0\ldots v_{t-1}w_{t}\ldots w_k] \right \} \\
&= -  \left\{ 
\sum\limits_{t=0}^{k-1} 
\sum\limits_{j=t}^{k} (-1)^{t+j} [v_0 \ldots v_t w_t \ldots \hat w_j \ldots w_k]  +
\sum\limits_{t=1}^{k} 
\sum\limits_{j=0}^{t} (-1)^{t+j+1}  [v_0 \ldots \hat v_j \ldots v_t w_t \ldots w_k]
\right \} \\
&= -  \left\{ 
\sum\limits_{t=1}^{k-1} 
\sum\limits_{j=t}^{k} (-1)^{t+j} [v_0 \ldots v_t w_t \ldots \hat w_j \ldots w_k]  +
\sum\limits_{t=1}^{k-1} 
\sum\limits_{j=0}^{t} (-1)^{t+j+1}  [v_0 \ldots \hat v_j \ldots v_t w_t \ldots w_k]
\right . \\
&\hspace{1cm} + \left .
\sum\limits_{j=0}^{k} (-1)^{j} [v_0 w_0 \ldots \hat w_j \ldots w_k]  +
+
\sum\limits_{j=0}^{k} (-1)^{k+j+1}  [v_0 \ldots \hat v_j \ldots v_k w_k]
\right \}.
\end{align*}

\end{proof}

\end{document}